\documentclass{amsart}

\usepackage{xcolor}

\newtheorem{thm}{Theorem}[section]
\newtheorem{lemma}[thm]{Lemma}
\newtheorem {proposition} [thm]{Proposition}
\newtheorem {corollary} [thm]{Corollary}

\theoremstyle{definition}
\newtheorem{definition}[thm]{Definition}

\theoremstyle{remark}
\newtheorem{remark}[thm]{Remark}

\numberwithin{equation}{section}

\newcommand{\al}{\alpha}

\newcommand{\clk}{\sigma}
\newcommand{\mbb}{\mu}
\newcommand{\plh}{PM}

\newcommand{\diin}{\theta}
\newcommand{\bdd}{b}
\newcommand{\ubb}{V_b}
\newcommand{\ier}{I}
\newcommand{\iker}{K}
\newcommand{\krr}{k}
\newcommand{\uer}{V_I}
\newcommand{\dbr}{\mathcal{H}}
\newcommand{\dom}{\mathcal{D}}
\newcommand{\ddo}{\partial\mathcal{D}}

\newcommand{\za}{\zeta}
\newcommand{\ph}{\varphi}
\newcommand{\hol}{\mathcal{H}ol}
\newcommand{\Rl}{\mathrm{Re\,}}
\newcommand{\Dbb}{\mathbb D}
\newcommand{\Tbb}{\mathbb T}
\newcommand{\dn}{{\mathbb D}^n}
\newcommand{\bn}{B_n}
\newcommand{\spn}{S_n}

\newcommand{\mn}{\Sigma}

\newcommand{\Nbb}{\mathbb N}

\newcommand{\kla}{I^*(H^2)}

\newcommand{\sph}{\partial B}

\begin{document}

\title[Clark measures]
{Clark measures and de Branges--Rovnyak spaces in several variables}


\author{Aleksei B.\ Aleksandrov}
\address{Department of Mathematics and Computer Science,
St.~Petersburg State University,
Line 14th (Vasilyevsky Island), 29, St.~Petersburg 199178, Russia}
\email{alex@pdmi.ras.ru}

\author{Evgueni Doubtsov}
\address{Department of Mathematics and Computer Science,
St.~Petersburg State University,
Line 14th (Vasilyevsky Island), 29, St.~Petersburg 199178,
Russia}
\email{dubtsov@pdmi.ras.ru}
\thanks{The research on Sections~2 and 3 was supported by Russian Science Foundation (grant No.~20-61-46016);
the research on Sections~1 and 4 was supported by Russian Science Foundation (grant No.~19-11-00058).}


\date{}

\dedicatory{}

\begin{abstract}
Let $B_n$ denote the unit ball of $\mathbb{C}^n$, $n\ge 1$, and let
$\mathcal{D}$ denote a finite product of $B_{n_j}$, $j\ge 1$.
Given a non-constant holomorphic function $b: \mathcal{D} \to B_1$, we study the corresponding family
$\sigma_\alpha[b]$, $\alpha\in\partial B_1$,
of Clark measures on the distinguished boundary $\partial\mathcal{D}$.
We construct a natural unitary operator from the de Branges--Rovnyak space $\mathcal{H}(b)$
onto the Hardy space $H^2(\sigma_\alpha)$.
As an application, for $\mathcal{D}= B_n$ and an inner function $I: B_n \to B_1$, we show
that the property $\sigma_1[I]\ll\sigma_1[b]$ is
directly related to the membership of an appropriate explicit function
in $\mathcal{H}(b)$.
\end{abstract}

\maketitle

\section{Introduction}\label{s_int}
Let $B_n$ denote the open unit ball of $\mathbb{C}^n$, $n\ge 1$,
and let $\sph_n$ denote the unit sphere.
We also use symbols $\Dbb$ and $\Tbb$ for the unit disc $B_1$ and the unit circle $\sph_1$, respectively.

Given $k\in \Nbb$ and $n_j\in \Nbb$, $j=1,2,\dots, k$, let
\[
\dom = \dom[n_1, n_2, \dots, n_k] = B_{n_1}\times B_{n_2}\cdots \times B_{n_k}  \subset \mathbb{C}^{n_1+ n_2 +\dots + n_k}.
\]
Model examples of $\dom$ are $B_n$ and the polydisc $\mathbb{D}^n$.

Let $C(z, \za) = C_\dom(z, \za)$ denote the Cauchy kernel for $\dom$.
Recall that
\[
  C_{\bn}(z, \za) = \frac{1}{(1- \langle z, \za \rangle)^{n}}, \quad z\in\bn,\ \za\in\spn.
\]
Let $\ddo$ denote the distinguished boundary
$\partial B_{n_1}\times \partial B_{n_2} \cdots \times \partial B_{n_k}$ of $\dom$.
Then
\[
  C_{\dom}(z, \za) =\prod_{j=1}^{k} \frac{1}{(1- \langle z_j, \za_j \rangle)^{n_j}},
  \ \ z=(z_1, z_2, \dots, z_k)\in\dom,\ \za=(\za_1, \za_2, \dots, \za_k)\in\ddo,
\]
where $z_j = (z_{j, 1}, z_{j,2}, \dots, z_{j, n_j}) \in B_{n_j}$ and
$\za_j = (\za_{j, 1}, \za_{j,2}, \dots, \za_{j, n_j}) \in \partial B_{n_j}$.

The corresponding Poisson type kernel is given by the formula
\[
P(z, \za) = \frac{C(z, \za) C(\za, z)}{C(z, z)},
\quad z\in\dom,\ \za\in\ddo.
\]
For $\dom=B_n$, $P(\cdot, \cdot)$ is often called the M\"obius invariant Poisson kernel;
see \cite{Ru80} for further details.

Let $M(\ddo)$ denote the space of complex Borel measures on $\ddo$.
For $\mu\in M(\ddo)$, the Cauchy transform $\mu_+$ is defined as
\[
  \mu_+(z) = \int_{\partial\dom} C(z, \za)\, d\mu(\za), \quad z\in \dom.
\]

\subsection{Clark measures}
Given an $\alpha\in\Tbb$ and a holomorphic function $\bdd: \dom\to \Dbb$, the quotient
\[
\frac{1-|\bdd(z)|^2}{|\al-\bdd(z)|^2}= \Rl \left(\frac{\al+ \bdd(z)}{\al- \bdd(z)} \right), \quad z\in \dom,
\]
is positive and pluriharmonic.
Therefore, there exists a unique positive measure $\clk_\al= \clk_\al[\bdd] \in M(\ddo)$
such that
\[
P[\clk_\al](z) = \Rl \left(\frac{\al+ \bdd(z)}{\al- \bdd(z)} \right), \quad z\in \dom.
\]
After the seminal paper of Clark \cite{Cl72},
various properties and applications of the measures $\clk_\al$ on the unit circle $\Tbb$ have been obtained;
see, for example, reviews \cite{GR15, MM06, PS06, Sa07} for further references.
Several results related to the Clark measures on the unit sphere $\partial B_n$, $n\ge 2$, are given in \cite{AD20}.

\subsection{Model spaces and de Branges--Rovnyak spaces}
Let $\mn$ denote the normalized Lebesgue measure on $\ddo$.

\begin{definition}\label{d_inner}
A holomorphic function $I:\dom \to \Dbb$ is called \textsl{inner}
if $|I(\za)|=1$ for $\mn$-a.e. $\za\in\ddo$.
\end{definition}

In the above definition,
$I(\za)$ stands, as usual, for
$\lim_{r\to 1-} I(r\za)$.
Recall that the corresponding limit exists $\mn$-a.e.
Also, by the above definition, every inner function is non-constant.

Given an inner function $I$ in $\dom$, we have
\[
P[\clk_\al](\za) =\frac{1-|I(\za)|^2}{|\al-I(\za)|^2}=0 \quad \mn\textrm{-a.e.},
\]
therefore, $\clk_\al = \clk_\al[I]$ is a singular measure. Here and in what follows, this means that
$\clk_\al$ and $\mn$ are mutually singular; in brief, $\clk_\al \bot\mn$.

Let $\hol(\dom)$ denote the space of holomorphic functions in $\dom$.
For $0<p<\infty$, the classical Hardy space $H^p=H^p(\dom)$ consists of those $f\in \hol(\dom)$ for which
\[
\|f\|_{H^p}^p = \sup_{0<r<1} \int_{\ddo} |f(r\za)|^p\, d\mn(\za) < \infty.
\]
As usual, we identify the Hardy space $H^p(\dom)$, $p>0$, and the space
$H^p(\ddo)$ of the corresponding boundary values.

For an inner function $\diin$ on $\Dbb$, the classical
model space $K_\diin$ is defined as
\[
K_\diin = H^2(\Tbb)\ominus \diin H^2(\Tbb).
\]
Clark \cite{Cl72} introduced and studied a family of
useful unitary operators
$U_\al : K_\diin \to L^2(\clk_\al)$, $\al\in\Tbb$.

For an inner function $I$ in $\dom$, there are several reasonable generalizations of $K_\diin$.
Consider the following direct analog of $K_\diin$:
\[
  \kla = H^2 \ominus I H^2.
\]
For $\dom = \bn$, it is shown in \cite[Theorem~5.1]{AD20} that
Clark's construction appropriately extends to $\kla$ and also provides natural unitary operators
$T_\al: \kla\to L^2(\clk_\al)$, $\al\in\Tbb$.

Observe that
\[
C(\za, z) = (1- \overline{I(z)} I(\za))C(\za, z)
+ \overline{I(z)} I(\za) C(\za, z)
\in \kla\oplus I H^2
\]
as functions of $\za$. Therefore,
\[
K(z, \za)
= (1-I(z)\overline{I(\za)}) C(z, \za)
\]
is the reproducing kernel
for the Hilbert space $\kla$ at $z\in\dom$,
that is,
\[
g(z) = \int_{\ddo} g(w) K(z, w) \, d\mn(w), \quad z\in\dom,
\]
for all $g\in \kla$.

Now, let $\bdd: \dom\to \Dbb$ be an arbitrary non-constant holomorphic function.
Direct inspection shows that the function
\[
k_b(z, w) = (1-b(z)\overline{b(w)}) C(z, w)
\]
has the reproducing kernel properties.
The corresponding Hilbert space $\dbr(b)\subset H^2$ is
called a de Branges--Rovnyak space.
In particular, $\kla = \dbr(I)$ for an inner function $I$.
Further details are given in \cite[Chapter~II]{Sa94} for $\Dbb$ in the place of $\dom$.

For $\al\in\Tbb$, Sarason \cite[Section~III-7]{Sa94} introduced unitary operators
\[
U_{b, \al}: \dbr(b) \to H^2(\clk_\al[\bdd])
\]
and closely related partial isometries
\[
V_{b, \al}: L^2(\clk_\al[\bdd]) \to \dbr(b),
\]
where $\dbr(b)\subset H^2(\Dbb)$ is the de Branges--Rovnyak space generated by $\bdd$,
$H^2(\clk_\al[\bdd])$ is a Hardy type space.
In the present paper, we construct analogous natural operators
$U_{b, \al}$ and $V_{b, \al}$, $\al\in \Tbb$,
for a given non-constant holomorphic function $\bdd: \dom\to \Dbb$; see Theorems~\ref{t_kla} and \ref{t_Vba_df}.

\subsection{Comparison of Clark measures}
Sarason \cite[Section~III-11]{Sa94} argued with the help of $V_{b, \al}$
to compare Clark measures on the unit circle $\Tbb$.
To show that the operators $V_{b, \al}$ are useful in combination with appropriate results in several complex variables,
we obtain the following comparison theorem for the Clark measures on the unit sphere $\sph_n$.

\begin{thm}\label{t_tal_on}
Let $I$ be an inner function in $\bn$ and let $\bdd:\bn\to\Dbb$, $n\ge 2$, be a non-constant
holomorphic function. Let $\clk=\clk_\al[I]$ and $\mu=\clk_\al[\bdd]$, $\al\in\Tbb$, be the corresponding Clark measures
and let $\iker_w(\cdot) = \iker(\cdot, w)$, where $\iker(z, w)$ denotes the reproducing kernel for $\kla$.
Then the following properties are equivalent:
\begin{enumerate}
  \item[(i)] $\clk\ll\mbb$ and $\frac{d\clk}{d\mbb}\in L^2(\mbb)$;
  \item[(ii)] the function
 \[
 \frac{\al-\bdd}{\al-\ier} \iker_w
 \]
 is in the de Branges--Rovnyak space $\dbr(\bdd)$ for all $w\in\bn$;
  \item[(iii)] the function
 \[
 \frac{\al-\bdd}{\al-\ier} \iker_w
 \]
 is in $\dbr(\bdd)$ for some $w\in\bn$.
\end{enumerate}
\end{thm}

\subsection{Organization of the paper}
Auxiliary properties of Clark measures are obtained in Section~\ref{s_aux_clk}.
Operators $U_{b, \al}$ and $V_{b, \al}$ are constructed in Section~\ref{s_clkdBR}.
The final Section~\ref{s_prf} is devoted to the proof of Theorem~\ref{t_tal_on}.

Theorem~\ref{t_tal_on} was announced in extended abstract \cite{ADbcn}.

\section{Cauchy integrals and Clark measures}\label{s_aux_clk}
The following lemma is a particular case of Exercise~1
from \cite[Chapter~8]{KrBook}.

\begin{lemma}\label{l_diag}
Let $F$ be a holomorphic function on $\dom\times \dom$.
If $F(z, \overline{z}) =0$ for all $z\in\dom$,
then $F(z, w) =0$ for all $(z, w) \in \dom\times \dom$.
\end{lemma}

The following key technical proposition is obtained in \cite[Proposition~3.5]{AD20} for $\dom=\bn$.

\begin{proposition}\label{p_cauchy_dbl}
Let $\bdd: \dom\to\Dbb$ be a holomorphic function and
let $\clk_\al = \clk_\al[\bdd]$, $\al\in \Tbb$, be the corresponding Clark measure.
 Then
  \[
  \int_{\ddo} C(z, \za) C(\za, w)\, d\clk_\al(\za) =
  \frac{1- \bdd(z)\overline{\bdd(w)}}{(1-\overline{\al}{\bdd(z)})(1-\al\overline{\bdd(w)})} C(z,w)
  \]
for all $\al\in\Tbb$, $z, w \in\dom$.
\end{proposition}
\begin{proof}
  The equality
  \[
  \int_{\ddo} P(z, \za) \, d\clk_\al(\za) = \frac{1-|\ph(z)|^2}{|\al- \ph(z)|^2}, \quad z\in \dom,
  \]
  and the definition of $P(z,\za)$ provide
   \[
  \int_{\ddo} C(z, \za) C (\za, z)\, d\clk_\al(\za) = \frac{1-|\ph(z)|^2}{|\al- \ph(z)|^2} C(z,z), \quad z\in \dom.
  \]
  It remains to apply Lemma~\ref{l_diag}.
\end{proof}

For $\mu\in M(\ddo)$, recall that $\mu_+$ denotes the Cauchy transform of $\mu$.

\begin{corollary}\label{c_cauchyof_clk}
Let $\ph: \dom\to\Dbb$, $d\ge 2$, be a holomorphic function
and let $\clk_\al = \clk_\al[\ph]$, $\al\in \Tbb$.
Then
\[
(\clk_\al)_+(z) = \frac{1}{1-\overline{\al} \ph(z)} + \frac{\al\overline{\ph(0)}}{1-\al\overline{\ph(0)}}
\]
for all $\al\in\Tbb$, $z\in\dom$.
\end{corollary}
\begin{proof}
  Apply Proposition~\ref{p_cauchy_dbl} with $w=0$.
\end{proof}

\section{Clark measures and de Branges--Rovnyak spaces}\label{s_clkdBR}
In this section, $b:\dom\to\Dbb$ is an arbitrary non-constant holomorphic function.

\subsection{Unitary operators $U_{\bdd, \al}: \dbr(\bdd)\to H^2(\clk_\al[b])$}
Fix an $\al\in\Tbb$.
Let $\clk_\al =\clk_\al[b]$ and let $\krr_w(\cdot) = \krr_\bdd(\cdot, w)$,
where $\krr_\bdd(z, w)$ denotes the reproducing kernel for $\dbr(b)$.
Define
\[
(U_{\bdd, \al} \krr_w)(\cdot) = (1-\al \overline{\bdd(w)}) C(\cdot, w),\quad w\in\dom.
\]
Let $H^2(\clk_\al)$ denote the closed linear span of the holomorphic polynomials or, equivalently,
of $C(\cdot, w)$, $w\in\dom$, in $L^2(\clk_\al)$.
In other words, $H^2(\clk_\al)$ is the Hardy space generated by $\clk_\al$.

\begin{thm}\label{t_kla}
For each $\al\in\Tbb$, $U_{\bdd, \al}$ has a unique extension
to a unitary operator from $\dbr(\bdd)$ onto $H^2(\clk_\al)$.
\end{thm}
\begin{proof}
Fix an $\al\in\Tbb$.
Applying Proposition~\ref{p_cauchy_dbl}, we obtain
  \begin{align*}
    (U_{\bdd, \al} \krr_w, U_{\bdd, \al} \krr_z)_{L^2(\clk_\al)}
    &=\int_{\ddo} (1-\al \overline{\bdd(w)}) C(\za, w) (1-\overline{\al} \bdd(z)) C(z, \za)\, d\clk_\al(\za) \\
    &=(1-\al \overline{\bdd(w)})(1-\overline{\al} \bdd(z)) \int_{\ddo} C(\za, w) C(z, \za)\, d\clk_\al(\za) \\
    &= (1- \bdd(z) \overline{\bdd(w)}) C(z, w)\\
    &= \krr_\bdd(z,w) = (\krr_w, \krr_z)_{\dbr(\bdd)}.
  \end{align*}
So, $U_{\bdd, \al}$ extends to an isometric embedding of $\dbr(\bdd)$ into $L^2(\clk_\al)$.
Since the linear span of the family $\{\krr_w\}_{w\in\dom}$ is dense in $\dbr(\bdd)$,
the extension is unique.
Finally, $U_{\bdd, \al}$ maps $\dbr(\bdd)$ onto $H^2(\clk_\al)$
by the definition of $H^2(\clk_\al)$.
\end{proof}

\subsection{Partial isometries $V_{\bdd, \al}: L^2(\clk_\al[b]) \to \dbr(\bdd)$}
Define
\begin{equation}\label{e_Vba_cauchy}
  (V_{\bdd, \al} g)(z) = (1-\overline{\al}\bdd(z))(g\clk_\al)_+(z), \quad g\in L^2(\clk_\al),\ z\in\dom.
\end{equation}

\begin{thm}\label{t_Vba_df}
For each $\al\in\Tbb$,
formula~\eqref{e_Vba_cauchy} defines a partial isometry from $L^2(\clk_\al)$ into $\dbr(b)$.
The restriction of $V_{\bdd, \al}$ to $H^2(\clk_\al)$ coincides with $U_{b, \al}^*$; in particular,
\begin{equation}\label{e_kernels}
V_{\bdd, \al} C(\cdot, w)(z) = (1-\al \overline{b(w)})^{-1} \krr_\bdd (z,w),
\end{equation}
where $\krr_\bdd(z, w)$ denotes the reproducing kernel for $\dbr(b)$.
\end{thm}
\begin{proof}
For $g(\za) = (1-\al \overline{\bdd(w)}) C(\za, w)$ with $w\in\dom$,
the definition of $U_{\bdd, \al}$ and Proposition~\ref{p_cauchy_dbl} guarantee that
\begin{equation}\label{e_Ual_cauchy}
  (U_\al^* g)(z) = (1-\overline{\al}\bdd(z))(g\clk_\al)_+(z), \quad z\in\dom.
\end{equation}
Therefore, \eqref{e_Ual_cauchy} holds for all $g\in H^2(\clk_\al)$;
hence, the restriction of $V_{\bdd, \al}$ on $H^2(\clk_\al)$
coincides with $U^*_\al$.
If $h\in L^2(\clk_\al)$ and $h\bot H^2(\clk_\al)$, then $(h\clk_\al)_+=0$.
Therefore, $V_{\bdd, \al}$ maps $L^2(\clk_\al)$ into $\dbr(b)$,
as required.
\end{proof}

\subsection{Cauchy transforms of $f\clk_\al$ with $f\in H^2(\clk_\al)$}
The following proposition is probably of independent interest.

\begin{proposition}\label{p_cau_zero}
Let $\al\in\Tbb$ and $b:\dom\to \Dbb$ be a non-constant holomorphic function.
Then $(f\clk_\al)_+ \not\equiv 0$ for any $f\in H^2(\clk_\al)\setminus \{0\}$.
\end{proposition}
\begin{proof}

Let $\al\in\Tbb$ and $f\in H^2(\clk_\al)\setminus \{0\}$.
Since $\krr_\bdd(z, \cdot) \in \overline{\dbr(\bdd)}$ and $U_\al = U_{b, \al}: \dbr(\bdd)\to H^2(\clk_\al)$ is unitary,
we obtain
\[
\begin{aligned}
U_\al^* f(z)
&= \int_{\ddo} U_\al^* f(\za) \krr_\bdd(z, \za)\, d\mn(\za) \\
&= \int_{\ddo} (1- \overline{\al} I(z)) f(\za) C(z, \za)\, d\clk_\al(\za) \\
&= (1- \overline{\al} I(z)) (f\clk_\al)_+(z)
\end{aligned}
\]
for $z\in\dom$.
Now, assume that $(f\clk_\al)_+ \equiv 0$. Then $U_\al^* f\equiv 0$, hence, $f=0$.
This contradiction ends the proof.
\end{proof}

For an inner function $b$, it is natural to ask whether $U_{b, \al}: \dbr(b) \to L^2(\clk_\al)$ is surjective or,
equivalently, $H^2(\clk_\al[b]) = L^2(\clk_\al[b])$.

\begin{corollary}\label{c_ual_onto}
Let $\al\in\Tbb$ and $I: \dom\to \Dbb$ be an inner function.
Then the following properties are equivalent:
\begin{enumerate}
  \item[(i)] $U_{I, \alpha}$ maps $\dbr(b)$ onto $L^2(\clk_\al)$;
  \item[(ii)] if $f\in L^2(\clk_\al)$ and $(f\clk_\al)_+ \equiv 0$, then $f=0$.
\end{enumerate}
\end{corollary}
\begin{proof}
(ii)$\Rightarrow$(i)
Assume that (i) does not hold, that is, $H^2(\clk_\al)\neq L^2(\clk_\al)$.
Then the definition of $H^2(\clk_\al)$ guarantees that there exists $f\in L^2(\clk_\al)\setminus \{0\}$ such that
\[
\int_{\ddo} f(\za) \overline{C(\za, z)}\, d\clk_\al(\za) =0
\]
for all $z\in \dom$.
In other words, $(f\clk_\al)_+ \equiv 0$ and we arrive to a contradiction.

By Proposition~\ref{p_cau_zero}, (i) implies (ii), hence, the proof is finished.
\end{proof}

\begin{remark}
Corollary~\ref{c_ual_onto} reduces to Proposition~\ref{p_cau_zero}
for $\dom=B_n$. Indeed, we have $H^2(\clk_\al[I]) = L^2(\clk_\al[I])$
for any $\al\in\Tbb$ and any inner function $I$ in the unit ball $B_n$, $n\ge 1$; see \cite{AD20}.
However, this is not the case for many inner functions in the polydisc $\dn$, $n\ge 2$.
To give a simple example, consider the following inner function: $I(z) =z_1$, $z\in\dn$, $n\ge 2$.
We have
\[
\clk_\al = \delta_\al(\za_1) \otimes m(\za_2) \otimes\dots \otimes m(\za_n),
\]
where $m$ denotes the normalized Lebesgue measure on $\Tbb$.
For all $\al\in\Tbb$, the space $H^2(\clk_\al)$ is strictly smaller than
$L^2(\clk_\al)$.
\end{remark}

\section{Proof of Theorem~\ref{t_tal_on}}\label{s_prf}

\subsection{Auxiliary results and definitions}
\subsubsection{Pluriharmonic measures}
A measure $\mu\in M(\ddo)$ is called \textsl{pluriharmonic} if the Poisson integral
\[
P[\mu](z) = \int_{\ddo} P(z, \za)\, d\mu(\za), \quad z\in\dom,
\]
is a pluriharmonic function.
Let $\plh(\dom)$ denote the set of all pluriharmonic measures.
Clearly, every Clark measure is an element of $\plh(\dom)$.

\subsubsection{Totally singular measures}
By definition, the ball algebra $A(\bn)$ consists of those $f\in C(\overline{\bn})$ which are holomorphic in $\bn$.
Let $M_0(\spn)$ denote the set of those probability measures $\rho\in M(\spn)$
which represent the origin for $A(\bn)$, that is,
\[
\int_{\spn} f\, d\rho = f(0) \quad \textrm{for all}\ f\in A(\bn).
\]
Elements of $M_0(\spn)$ are called representing measures.

\begin{definition}\label{d_totnull}
A measure $\mu\in M(\spn)$ is said to be \textsl{totally singular} if $\mu\bot \rho$
for all $\rho\in M_0(\spn)$.
\end{definition}

\begin{proposition}[{\cite[Theorem~10]{Dzap94}}]\label{t_aif}
Let $\mu\in\plh(\spn)$. Then $\mu^s$ is totally singular.
\end{proposition}

\begin{remark}
For positive pluriharmonic measures, the above theorem was obtained in \cite[Chapter~5, Section~3.3.3]{Aab85}.
In fact, we will apply Proposition~\ref{t_aif} to Clark measures, that is, to positive $\mu\in\plh(\spn)$.
\end{remark}

\subsubsection{Henkin measures}
\begin{definition}[see {\cite[Section~9.1.5]{Ru80}}]\label{d_Henkin}
We say that $\mu\in M(\spn)$ is a \textsl{Henkin measure} if
\[
\lim_{j\to\infty} \int_{\spn} f_j\, d\mu =0
\]
for any bounded sequence $\{f_j\}_{j=1}^\infty \subset A(\bn)$ with the following property:
\[
\lim_{j\to\infty} f_j(z) = 0 \quad\textrm{for any\ } z\in\bn.
\]
\end{definition}

\subsection{Proof of Theorem~\ref{t_tal_on}}
We are given an inner function $I$ in $\bn$ and a non-constant holomorphic function
$b: \bn\to \Dbb$.
Without loss of generality, assume that $\al=1$.
So, let $\clk=\clk_1[I]$, $\mu = \clk_1[b]$ and $\iker_w = \iker(\cdot, w)$,
where $\iker(z, w)$
denotes the reproducing kernel for $I^*(H^2) = \dbr(I)$.
Applying formula~\eqref{e_Vba_cauchy} and property~\eqref{e_kernels} with $\bdd =I$, we obtain
\begin{equation}\label{e_by_uer}
(1-b) (C(\cdot, w)\clk)_+ = \frac{1-\bdd}{1-\ier} \uer C(\cdot, w)
= ( 1- \overline{\ier(w)})^{-1} \frac{1-\bdd}{1-\ier} \iker_w.
\end{equation}

Now, we are in position to prove the theorem.

(i)$\Rightarrow$(ii)
By the definition of $\ubb$, we have
\[
\ubb\left( \frac{d\clk}{d\mbb} C(\cdot, w)\right) =
(1-b) \left(\frac{d\clk}{d\mbb}C(\cdot, w) \mbb\right)_+
= (1-b) (C(\cdot, w)\clk)_+.
\]
Since $\ubb$ maps into $\dbr(\bdd)$, combination of the above property and \eqref{e_by_uer}
provides (ii).

(ii)$\Rightarrow$(iii) This implication is trivial.

(iii)$\Rightarrow$(i)
By assumption, we are given a point $w\in\bn$ and a function $q = q_w \in L^2(\mbb)$ such that
\[
( 1- \overline{\ier(w)})^{-1} \frac{1-\bdd}{1-\ier} \iker_w = \ubb q
= (1-b)(q\mbb)_+.
\]
The above property and \eqref{e_by_uer} guarantee that
\[
(C(\cdot, w)\clk- q\mbb)_+ =0.
\]
Hence,
\[
C(w, \cdot)\clk- \overline{q}\mbb
\]
is a Henkin measure.
Thus, by the Cole--Range Theorem (see \cite{CR72} or \cite[Theorem~9.6.1]{Ru80}),
there exists a representing measure $\rho$ such that
$C(\cdot, w)\clk- {q}\mbb \ll \rho$; in particular,
\begin{equation}\label{e_llrep}
C(\cdot, w)\clk- {q}\mbb^s \ll \rho\quad \textrm{for some}\ \rho\in M_0(\spn).
\end{equation}
Recall that $\clk$ is a singular measure.
Therefore, $\clk$ and $\mbb^s$ are totally singular
by Proposition~\ref{t_aif}. Hence,
\[
C(\cdot, w)\clk- {q}\mbb^s\ \textrm{is totally singular.}
\]
Combining this observation and \eqref{e_llrep},
we conclude that
\[
\clk= \frac{q}{C(\cdot, w)} \mbb^s.
\]
In particular, $\clk\ll\mbb$ and
$\frac{d\clk}{d\mbb}\in L^2(\mbb)$, as required.
The proof of Theorem~\ref{t_tal_on} is finished.

\end{document}